\RequirePackage{fix-cm}
\documentclass[smallextended]{svjour3}       
\smartqed  
\usepackage{graphicx}
\usepackage{amsbsy}
\usepackage{amsmath}
\usepackage{amsfonts}
\usepackage{mathrsfs}
\usepackage{mathtools}
\usepackage{bm}
\usepackage{amssymb}
\usepackage{bbm}
\usepackage{enumerate}

\usepackage{subfig}
\newcount\colveccount
\newcommand*\colvec[1]{
	\global\colveccount#1
	\begin{pmatrix}
		\colvecnext
	}
	\def\colvecnext#1{
		#1
		\global\advance\colveccount-1
		\ifnum\colveccount>0
		\\
		\expandafter\colvecnext
		\else
	\end{pmatrix}
	\fi
}

\newcommand{\ndN}{\mathbb{N}}
\newcommand{\ndZ}{\mathbb{Z}}

\newcommand{\ndR}{\mathbb{R}}

\renewcommand{\Pr}[1]{\mathbb{P}(#1)}

\newcommand{\Ex}[1]{\mathbb{E}[#1]}

\newcommand{\Exb}[1]{\mathbb{E}\left[#1\right]}

\newcommand{\convdis}{\,{\buildrel d \over \longrightarrow}\,}
\newcommand{\convd}{\,{\buildrel d \over \longrightarrow}\,}

\newcommand{\convp}{\,{\buildrel p \over \longrightarrow}\,}

\newcommand{\eqdist}{\,{\buildrel d \over =}\,}

\newcommand{\cM}{\mathcal{M}}

\newcommand{\mM}{\mathsf{M}}

\newcommand{\mfL}{\mathfrak{L}}

\newcommand{\mfX}{\mathfrak{X}}

\newcommand{\ve}{\mathrm{v}}

\begin{document}

\title{Quenched local convergence of Boltzmann planar maps}


\author{Benedikt Stufler}


\institute{Benedikt Stufler \at
              Vienna University of Technology \\
              Institute of
Discrete Mathematics and Geometry \\
E-mail: \includegraphics{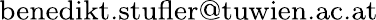}
}


\maketitle

\begin{abstract}
Stephenson~(2018) established annealed local convergence of Boltzmann planar maps conditioned to be large. The present work uses results on rerooted multi-type branching trees to prove a quenched version of this limit.
\end{abstract}

\section{Introduction}

A planar map $M$ is a connected planar
graph,  possibly  with  loops  and  multiple  edges,  together  with  an  embedding  into  the plane. Usually one edge is directed and distinguished as the root edge.  Various analytic, combinatorial, and probabilistic techniques for studying models of random planar maps have been developed, see~\cite{MR1871555,zbMATH06549737}. The bijection by \cite{MR2097335} encodes planar maps as mobiles, which are vertex-labelled $4$-type planar trees. This allows for a generating procedure for certain models of random planar maps using $4$-type Galton--Watson trees, see \cite{MR2509622}. For bipartite Boltzmann planar maps, a bijection constructed by \cite{MR3342658} simplifies the generating procedure to use only monotype Galton--Watson trees. However, it is an open problem whether a full reduction to mono-type trees is possible in the non-bipartite case, hence the need to study multi-type Galton--Watson trees for this purpose persists.\footnote{The author thanks Sigurdur {\"O}rn Stef{\'a}nsson for related comments.} 

Recent work by Stephenson \cite{MR3769811} establishes local convergence of conditioned regular critical multi-type Galton--Watson trees, and applies this convergence to a conditioned Boltzmann planar map $\mM_n$. The main application is a limit theorem that shows how an Infinite Boltzmann Planar Map $\hat{\mM}$ describes the asymptotic behaviour  of the vicinity of the root-edge of $\mM_n$ as $n \to \infty$. This generalizes local convergence results for bipartite maps by~\cite{MR3183575} and special cases like triangulations and quadrangulations by~\cite{MR2013797,2005math.....12304K}.

The present work establishes a corresponding \emph{quenched} version of the limit theorem. Roughly speaking, the difference is that instead of studying the probability for the vicinity of the root-edge of $\mM_n$ to have a certain shape, we establish laws of large numbers for the  number of corners, faces, and vertices whose vicinity has this shape. See Theorem~\ref{te:regcritquench}. Our main tools are quenched limits of rerooted multitype trees established recently in~\cite{multrerooted}.\footnote{The results of the present work were initially part of~\cite{multrerooted}. The paper was split during the review process following a referee's recommendation.}

As an application,  we deduce quenched local convergence of the random planar map $\mM_n^t$ with  $n$ edges and a positive weight $t>0$ at vertices.  That is, $\mM_n^t$ assumes a map $M$ with $n$ edges with probability proportional to $t^{\ve(M)}$, with $\ve(M)$ denoting the number of vertices of $M$. See Theorem~\ref{te:quenchmaptt}. 
 The vertex weighted random planar map $\mM_n^t$ is related to the study of uniform random planar graphs, see~\cite{MR3068033,MR2735332}. We  apply the quenched local convergence of $\mM_n^t$  in the subsequent paper~\cite{planar} to deduce local convergence of the uniform random planar graph.

\section*{Notation}
We let  $\ndN_0 = \{0,1,2,\ldots\}$ denote the collection of non-negative integers, and $\ndN$ the collection of positive integers. The law of a random variable $X: \Omega \to S$ with values in some measurable space $S$ is denoted by $\mfL(X)$.  If $Y: \Omega \to S'$ is a random variable with values in some measurable space $S'$, we let $\mfL(X \mid Y)$ denote the conditional law of $X$ given $Y$.  All unspecified limits are taken as $n \to \infty$.  Convergence in probability and distribution are denoted by $\convp$ and $\convd$. 
We say an event holds with high probability if its probability tends to $1$ as $n$ becomes large. 
For any sequence $(a_n)_{n \ge 1}$ of positive real numbers we let $O_p(a_n)$ denote a random variable $Z_n$ such that $(Z_n / a_n)_{n \ge 1}$ is stochastically bounded. 

\section*{Index of terminology}

\noindent The following list summarizes frequently used terminology.

\begin{tabular}{p{0.1\linewidth} p{0.75\linewidth}} 
	$\bm{\xi}$ & An unordered $D$-offspring distribution  $\bm{\xi} = (\bm{\xi}_i)_{i \in \mathfrak{G}}$, page~\pageref{de:offspring}.\\[2pt]
	$\#_i(\cdot)$ & Number of vertices of type $i \in \mathfrak{G}$, page~\pageref{de:numi}.\\[2pt]
	$|\cdot|_{\bm{\gamma}}$ & Sum of vertices weighted depending on their type, page~\pageref{eq:weighted}\\[2pt]	
	$\bm{T}(\eta)$ & A $\bm{\xi}$-Galton--Watson tree with (possilby random) root type $\eta$, page~\pageref{de:bmt}.\\[2pt]
	$\bm{T}^\kappa$ & Like $\bm{T}(\kappa)$, but non-root vertices of type $\kappa$ receive no offspring, page~\pageref{de:ta}.
\end{tabular}

\begin{tabular}{p{0.1\linewidth} p{0.75\linewidth}} 
	$\hat{\bm{T}}^\kappa$ & A random tree with a marked leaf of type~$\kappa$. Distributed like  $\bm{T}^\kappa$ biased by the number of vertices with type~$\kappa$, page~\pageref{eq:sbstep}.\\[2pt]
	$\hat{\bm{T}}(\kappa)$ & A random infinite tree with a marked vertex of type~$\kappa$ and a spine that grows backwards, page~\pageref{de:htk}.\\[2pt]
	$\hat{\bm{T}}^{\kappa,\iota}$ & A random tree with root type~$\kappa$ and a marked vertex of type $\iota$. Obtained by biasing $\bm{T}^{\kappa}$ by the number of vertices of type~$\iota$, page~\pageref{de:tkg}.\\[2pt]
	$\hat{\bm{T}}(\kappa, \iota)$ & A random infinite tree with a marked vertex of type~$\iota$ and a spine that grows backwards, page~\pageref{de:tkg2}.
\end{tabular}        


\vspace{1 \baselineskip}

\section{Preliminaries}

\subsection{Local topologies for planar maps}

The local topology describes how similar two planar maps are in the vicinity of specified root vertices or root edges.  We briefly recall relevant notions and refer the reader to the elegant presentations by \cite{curienlecture2} for details.

Let $\mathfrak{M}^{\mathrm{e}}$ denote the collection of  finite planar maps with an oriented root edge.  The origin of this root edge is called the root vertex. The face to the left of the oriented root-edge is called the root face. 
Likewise, we let  $\mathfrak{M}^{\mathrm{v}}$ denote the collection of finite planar maps that only have a specified root vertex instead of an oriented root edge. We also let $\mathfrak{M}^{\mathrm{f}}$ denote the collection of  finite planar maps that only carry a marked root face instead. 
In the following, $\mathfrak{M}$ refers to $\mathfrak{M}^{\mathrm{e}}$, $\mathfrak{M}^{\mathrm{v}}$, or $\mathfrak{M}^{\mathrm{f}}$, as all related concepts are analogous for these three cases. Note that $\mathfrak{M}$ is countably infinite.

For any integer $k \ge 0$ we may consider the subset $\mathfrak{M}_k \subset \mathfrak{M}$ of planar maps where each vertex has ``distance'' at most $k$ from the root. Here ``distance from the root'' refers to the graph distance from the root vertex in case of vertex rooted maps, or the graph distances to the ends of the root-edge for edge rooted maps. For face rooted maps, we define the  distance as the length of some shortest path from the vertex to the boundary of the marked face. We equip $\mathfrak{M}_k$ with the discrete topology.

The projection \[
U_k: \mathfrak{M} \to \mathfrak{M}_k\] maps a planar map $M$ to the $k$-neighbourhood  $U_k(M)$  of its root. Depending on whether $\mathfrak{M}$ refers to $\mathfrak{M}^{\mathrm{e}}$, $\mathfrak{M}^{\mathrm{v}}$, or $\mathfrak{M}^{\mathrm{f}}$ we view  $U_k(M)$ as equipped with an oriented root edge, a root vertex, or a root face.

The local topology on the collection $\mathfrak{M}$ is the coarsest topology that makes these projections continuous. This projective limit topology is metrizable by
\[
d_{\mathrm{loc}}(M, M') = \frac{1}{1 + \sup\{k \ge 0 \mid U_k(M) = U_k(M')\}}, \qquad M, M' \in \mathfrak{M}.
\]
The space $(\mathfrak{M}, d_{\mathrm{loc}})$ is not complete. One way to complete it, is to form the space $\overline{\mathfrak{M}}$ of coherent sequences
\[
\overline{\mathfrak{M}} = \{ (M_1, M_2, \ldots) \mid U_i(M_{i+1}) = M_i \text{ for all $i \ge 1$} \} \subset \mathfrak{M}^{\ndN}.
\]
We may interpret $\mathfrak{M}$ as a subset of $\overline{\mathfrak{M}}$, and extend $d_{\mathrm{loc}}$ and $U_k(\cdot)$ in a canonical way. This makes $\overline{\mathfrak{M}}$ a Polish space, see~\cite[Prop. 1]{curienlecture2} for details.

Let $(\mM_n)_{n \ge 1}$ be a sequence of random finite planar maps, and let $u_n$ be either a uniformly selected vertex, oriented edge, or face. This makes the pair $(\mM_n, u_n)$ a random element of~$\overline{\mathfrak{M}}$. Here we forget about any possibly present root of $\mM_n$, and only consider $u_n$ as the new root. Distributional convergence of $(\mM_n, u_n)_{n \ge 1}$ is equivalent to distributional convergence of each neighbourhood $U_k(\mM_n, u_n)$, $k \ge 0$, as $n$ tends to infinity. If $\hat{\mM}$ is a random element of~$\overline{\mathfrak{M}}$, then
\begin{align}
	(\mM_n, u_n) \convdis \hat{\mM}
\end{align}
is equivalent to
\begin{align}
	\Pr{U_k(\mM_n, u_n) = M} \to \Pr{U_k(\hat{\mM})=M}
\end{align}
for each fixed integer $k \ge 0$ and each finite planar map $M \in \mathfrak{M}$ as $n \to \infty$. Using language from statistical physics, this form of convergence is also called \emph{annealed} convergence.

The collection $\mathbb{M}_1(\overline{\mathfrak{M}})$  of probability measures on the Borel sigma algebra of $\overline{\mathfrak{M}}$ is a Polish space with respect to the weak convergence topology.   The conditional distribution $\mfL( (\mM_n, u_n) \mid \mM_n)$ is a random element of $\mathbb{M}_1(\overline{\mathfrak{M}})$. We say $(\mM_n, u_n)$ converges in the \emph{quenched} sense, if the random probability measure $\mfL( (\mM_n, u_n) \mid \mM_n)$ converges in distribution to a random element of $\mathbb{M}_1(\overline{\mathfrak{M}})$. For the special case where the limit is almost surely constant and given by the law $\mfL(\hat{\mM})$ of some random element $\hat{\mM}$ of $\overline{\mathfrak{M}}$, we say $(\mM_n, u_n)$ converges in the quenched sense towards~$\hat{\mM}$.

\subsection{Limits of rerooted multi-type trees}

Given an integer $D \ge 1$, a \emph{$D$-type plane tree} $T$ is a plane tree where we assign to each vertex a \emph{type} from $\{1, \ldots, D\}$.  In particular, $T$ has a root vertex, and for each vertex we have a linear order on its collection of children.  For $1 \le j \le D$ we let \label{de:numi} $\#_j T$ denote the total number of vertices with type $j$. Furthermore, for any vector $\bm{\gamma} = (\gamma_1, \ldots, \gamma_D) \in \ndR^D$ we set
\begin{align}
	\label{eq:weighted}
	|T|_{\bm{\gamma}} = \sum_{i=1}^D \gamma_i \#_i T.
\end{align}

If we distinguish a vertex $v$ of $T$, we form a \emph{marked} tree $(T,v)$. The subtree consisting of $v$ and all its descendants is called the \emph{fringe subtree} of $T$ at $v$. For any integer $k \ge 0$ we may form the \emph{extended fringe subtree} $f^{[k]}(T,v)$ consisting of the fringe subtree of $T$ at the $k$th ancestor of $v$, marked at the vertex corresponding to $v$. Of course, this only makes sense if $v$ has height at least $k$ in~$T$. Otherwise, we set  $f^{[k]}(T,v)$ to some placeholder value. 

The path from $v$ to the root of $T$ is called the \emph{spine} of the marked tree $(T,v)$. We may also consider marked trees where this spine has a countably infinite length, such that $v$ has a countably infinite number of ancestors. We let $\mfX$ denote the collection of all finite marked $D$-type trees and all marked $D$-type trees with an infinite spine such that all extended fringe subtrees are finite. 

The collection $\mfX$ may be endowed with a metric $d_{\mfX}$ such that for all $T^\bullet_1, T_2^\bullet \in \mfX$
\begin{align}
	d_{\mfX}(T^\bullet_1, T_2^\bullet) = 
	\begin{cases}
		2,  &f(T_1^\bullet) \neq f(T_2^\bullet) \\
		2^{- \sup\{k \ge 0 \mid f^{[k]}(T_1^\bullet) = f^{[k]}(T_2^\bullet)\}}, &f(T_1^\bullet) = f(T_2^\bullet). 
	\end{cases}
\end{align}
This makes $(\mfX, d_{\mfX})$ a Polish space, see \cite[Prop. 1]{multrerooted}.

Let $(\bm{T}_n)_{n \ge 1}$ be a sequence of random finite $D$-type trees. Let \[\mathfrak{G}_0 \subset \{1, \ldots, D\}\] denote a non-empty subset, such that the probability for $\bm{T}_n$ to have vertices with type in $\mathfrak{G}_0$ tends to $1$ as $n$ becomes large. Let $v_n$  be uniformly selected among all vertices of $\bm{T}_n$ with type in $\mathfrak{G}_0$. Then $(\bm{T}_n, v_n)$ is a random element of $\mfX$. We say $(\bm{T}_n, v_n)$ convergences in the \emph{annealed} sense towards a random element $\bm{T}^\bullet$ of $\mfX$, if
\[
(\bm{T}_n, v_n) \convdis \bm{T}^\bullet
\]
in the usual sense of distributional convergence of random elements of the Polish space~$\mathfrak{X}$. The  conditional distribution $\mfL( (\bm{T}_n, v_n) \mid \bm{T}_n)$ is a random element of the collection $\mathbb{M}_1(\mathfrak{X})$ of Borel probability measures on $\mathfrak{X}$. That is we take the tree $\bm{T}_n$ (this is where the randomness comes from) and consider the uniform distribution on all marked versions of $\bm{T}_n$ where the marked vertex has type in $\mathfrak{G}_0$.  We say $\bm{T}^\bullet$ is the \emph{quenched} limit of $(\bm{T}_n, v_n)$,~if
\[
\mfL( (\bm{T}_n, v_n) \mid \bm{T}_n) \convdis  \mfL(\bm{T}^\bullet)
\]
in the  sense of distributional convergence of random elements of the Polish space~$\mathbb{M}_1(\mathfrak{X})$.

\subsection{Galton--Watson trees}

Let $D \ge 1$ be an integer.
A $D$-type Galton--Watson tree is a random locally finite $D$-type plane tree defined as follows. Let \label{de:offspring} $\bm{\xi} = (\bm{\xi}_i)_{1 \le i \le D}$ be a family of random elements $ \bm{\xi}_i \in \ndN_0^D$. For any integer $1 \le \kappa \le D$ the $\bm{\xi}$-Galton--Watson  $\bm{T}(\kappa)$ starts with a single root vertex with type $\kappa$. For all $1 \le i \le D$ any vertex of type $i$ receives offspring vertices according to an independent copy of $\bm{\xi}_i$, with the $j$th coordinate (for $1 \le j \le D$) corresponding to the number of children with type~$j$. For our purposes, we will always assume that the collection of all children is ordered uniformly at random. If $\eta$ is a random element of $\{1, \ldots, D\}$, independent from all previously considered random variables, we let $\bm{T}(\eta)$ denote the mixture of $(\bm{T}(\kappa))_{1 \le \kappa \le D}$ that assumes $\bm{T}(\kappa)$ with probability $\Pr{\eta = \kappa}$ for each $ 1\le \kappa \le D$. That is, here the type of the root vertex is random and distributed like  $\eta$.\label{de:bmt}

We define $\bm{T}^\kappa$\label{de:ta} similar to $\bm{T}(\kappa)$, only that non-root vertices with type $\kappa$ receive no offspring. Let us assume that 
\begin{align}
	\bm{T}^\kappa \text{ is a.s. finite} \qquad \text{and} \qquad  \Exb{\#_\kappa \bm{T}^\kappa} = 2.
\end{align}
This allows us to define the $\kappa$-biased version $\hat{\bm{T}}^\kappa$ with distribution
\begin{align}
	\label{eq:sbstep}
	\Pr{\hat{\bm{T}}^\kappa = (T^\kappa, u)} = \Pr{\bm{T}^\kappa = T^\kappa}
\end{align}
for any pair $(T^\kappa, u)$ of a finite $\mathfrak{G}$-type tree $T^\kappa$ (with the root having type $\kappa$ and all non-root vertices of type $\kappa$ having no offspring) and a non-root leaf $u$ of $T^\kappa$ with type $\kappa$.

We construct a random tree $\hat{\bm{T}}(\kappa)$ that has an infinite ``backwards'' growing spine $u_0, u_1, \ldots$  of type $\kappa$ vertices, such that $u_{\ell + 1}$ is an ancestor (not necessarily parent) of $u_{\ell}$ for all $\ell \ge 0$. The construction is as follows. We start with the vertex $u_0$ that becomes the root of an independent copy of $\bm{T}(\kappa)$. The vertex $u_1$ becomes the root of an independent copy of $\hat{\bm{T}}^\kappa$, which has a marked leaf. All non-marked leaves of type $\kappa$ become roots of independent copies of $\bm{T}(\kappa)$, and we identify the marked leaf with $u_0$ (``glueing'' the two vertices together). We proceed in this way with an ancestor $u_2$ of $u_1$ and so on, yielding an infinite backwards growing spine $u_0, u_1, \ldots$ of type $\kappa$ vertices.

The tree  $\hat{\bm{T}}(\kappa)$ \label{de:htk} constitutes the multi-type analogue of Aldous' invariant sin-tree constructed in~\cite{MR1102319} for critical monotype Galton--Watson trees. The abbreviation \emph{sin} stands for \emph{single infinite path}.

Suppose that $\iota  \in \{1, \ldots, D\}$ is a type. If the number of non-root type $\iota$-vertices in $\bm{T}^\kappa$ has a finite non-zero expectation $E$, we may form the $\iota$-biased version $\bm{T}^{\kappa, \iota}$ of $\bm{T}^\kappa$ with distribution
\begin{align}
	\label{de:tkg}
	\Pr{\hat{\bm{T}}^{\kappa, \iota} = (T^\kappa, u)} = \Pr{\bm{T}^\kappa = T^\kappa} / E
\end{align}
for any pair $(T^\kappa, u)$ of a finite $\mathfrak{G}$-type tree $T^\kappa$ (with the root having type $\kappa$ and all non-root vertices of type $\kappa$ having no offspring) and a non-root leaf $u$ of $T^\kappa$ with type $\iota$. This allows us to construct the tree $\hat{\bm{T}}(\kappa, \iota)$\label{de:tkg2} analogous to $\hat{\bm{T}}(\kappa)$ with the only difference being that in the construction we start with a type $\iota$ vertex $u_0$ that becomes the root of an independent copy of $\bm{T}(\iota)$, and for $u_1$ we use $\bm{T}^{\kappa, \iota}$ instead of a copy of $\bm{T}^{\kappa}$. Hence $u_1, u_2, \ldots$ have type~$\kappa$, but $u_0$ has type~$\iota$.

\section{Boltzmann planar maps}

\label{sec:weigthedmaps}

We recall important background on Boltzmann planar maps~\cite{MR2509622} and the Bouttier--Di Francesco--Guitter transformation~\cite{MR2097335}. Our presentation follows closely that of~\cite[Sec. 5]{MR3769811}, with some additional emphasis in Section~\ref{sec:mobiles} on how the labels of a Boltzmann mobile may be constructed from conditionally independent choices for each vertex of the underlying Galton--Watson tree.

\subsection{The Boltzmann distribution on planar maps} The collection of all finite  planar maps with an oriented root edge and an additional marked vertex will be denoted by~$\cM$.  Throughout we let  $\bm{q} = (q_n)_{n \ge 1}$  denote a family of non-negative numbers such that $q_n >0$ for at least one $n \ge 3$. To any element $M \in \cM$ we  assign a weight
\begin{align}
	W_{\bm{q}}(M) = \prod_{f} q_{\mathrm{deg}(f)}.
\end{align}
Here the index $f$ ranges over the faces of the planar map $M$, and  $\mathrm{deg}(f)$ denotes the \emph{degree} of the face $f$. That is, $\mathrm{deg}(f)$ is  the number of half-edges on the boundary of the face $f$. (The reason why we count half-edges instead of edges is that an edge on the boundary has to be counted twice if both of its sides are incident to the face.) A weight-sequence $\bm{q}$ is said to be \emph{admissible}, if 
\begin{align}
	\label{eq:Zq}
	Z_{\bm{q}} := \sum_{M \in \cM} W_{\bm{q}}(M) < \infty.
\end{align}
In this case, we may form the \emph{Boltzmann distributed} (vertex marked) planar map $\mM$ with distribution given by
\begin{align}
	\Pr{\mM = M} = W_{\bm{q}}(M) / Z_{\bm{q}}, \qquad M \in \cM.
\end{align}
Likewise we may form analogously the Boltzmann planar map $\tilde{\mM}$ (and conditioned versions thereof) by using the class of maps without a marked vertex instead of $\cM$. Note that $\tilde{\mM}$ and $\mM$ follow different distributions, as $\mM$ is biased by the number of vertices.


\subsection{Mobiles obtained from branching processes}
\label{sec:mobiles}

A pointed map from $\cM$ is said to be positive, neutral, or negative, if the origin of the directed root edge is closer, equally far away, or farther away from the marked vertex than the destination of the root edge. We let $\cM^+$, $\cM^0$, and $\cM^-$ denote the corresponding subclasses of $\cM$, and form the sums $Z_{\bm{q}}^+$, $Z_{\bm{q}}^0$, and $Z_{\bm{q}}^-$ as in \eqref{eq:Zq}, but with the sum index constrained to the corresponding subclass.  For all $x,y\ge 0$ we define the bivariate series
\begin{align}
	\label{eq:do1}
	f^\bullet(x,y) &=  \sum_{k,k' \ge 0} \binom{2k + k'+1}{k+1} \binom{k+k'}{k} q_{2 + 2k + k'} x^k y^{k'}, \\
	\label{eq:do2}
	f^\diamond(x,y) &= \sum_{k, k' \ge 0} \binom{2k + k'}{k}  \binom{k+k'}{k} q_{1+2k+k'} x^k y^{k'}.
\end{align}
If the weight sequence $\bm{q}$ is admissible, we may define an irreducible $4$-type offspring distribution $\bm{\xi} = (\bm{\xi})_{1 \le i \le 4}$  as follows. Vertices of the first type produce a geometric number of vertices of the third type:
\begin{align}
	\Pr{\bm{\xi}_1 = (0,0,k,0)} = \frac{1}{Z_{\bm{q}}^+} \left( 1 - \frac{1}{Z_{\bm{q}}^+} \right)^k, \qquad k \ge 0.
\end{align}
Vertices of the second type always produce a single offspring vertex of the fourth type, that is
\begin{align}
	\Pr{\bm{\xi}_2 = (0,0,0,1)} = 1.
\end{align}
Vertices of the third and fourth type only produce offspring of the first or second type. Their coordinates ${\xi}_{3,1}, {\xi}_{3,2}$ and ${\xi}_{4,1}, {\xi}_{4,2}$ are determined  by
\begin{align}
	\label{eq:rr1}
	\Ex{x^{{\xi}_{3,1}} y^{{\xi}_{3,2}}}& = \frac{f^\bullet(x Z_{\bm{q}}^+, y\sqrt{Z_{\bm{q}}^0} )}{f^\bullet(Z_{\bm{q}}^+, \sqrt{Z_{\bm{q}}^0} )} \\
	\label{eq:rr2}
	\Ex{x^{{\xi}_{4,1}} y^{{\xi}_{4,2}}}& = \frac{f^\diamond(x Z_{\bm{q}}^+, y\sqrt{Z_{\bm{q}}^0} )}{f^\diamond(Z_{\bm{q}}^+, \sqrt{Z_{\bm{q}}^0} )}.
\end{align}
Here we have used that the denominators in~\eqref{eq:rr1} and~\eqref{eq:rr2} are finite. This follows from~\cite[Prop. 1]{MR2509622}, see Section~\ref{sec:regimes} below for details.

For a type $\kappa=1$ or $\kappa=2$ we consider the following sampling procedure. The result is a random $4$-type tree where the offspring is ordered and each vertex $v$ receives a label $\ell(v)$ with $\ell(v) \in \ndZ$ if $v$ has type $1$ or $3$, and $\ell(v) \in \frac{1}{2} + \ndZ$ otherwise.

\begin{enumerate}
	\item Consider the $\bm{\xi}$-Galton--Watson tree $\bm{T}(\kappa)$ that starts with a single vertex of type $\kappa$. We consider the offspring vertices as ordered in a uniformly selected manner.
	\item For each vertex $v$ of type $3$ or $4$ in $\bm{T}(\kappa)$ with outdegree $d \ge 1$ let $v_0$ denote its parent and let $v_1, \ldots, v_d$ denote its ordered offspring. For ease of notation, we set $v_{d+1} := v_0$. Note that $v_0, \ldots, v_d$ all have types in~$\{1,2\}$. Uniformly select a $(d+1)$-dimensional  vector \[\bm{\beta}_{\bm{T}(\kappa)}(v_0) =  (a_0, \ldots, a_{d})\] satisfying the following two conditions:
	\begin{enumerate}
		\item $\sum_{i=0}^d a_i = 0$.
		\item For all $0 \le i \le d$:
		\subitem If $v_i$ and $v_{i+1}$ both have type $1$, then $a_i \in \{-1, 0, 1, \ldots\}$.
		\subitem If $v_i$ and $v_{i+1}$ both have type $2$, then $a_i \in \{0, 1, 2, \ldots \}$.
		\subitem If $v_i$ and $v_{i+1}$ have different types, then $a_i \in \{-1/2, 1/2, 3/2, \ldots\}$.
	\end{enumerate}
	\item Assign to each vertex $v \in \bm{T}(\kappa)$ a label $\ell(v)$ in a unique way satisfying the following conditions.
	\begin{enumerate}
		\item The root of $\bm{T}(\kappa)$ receives label $0$ if it has type $1$ and label $1/2$ if it has type $2$.
		\item Vertices of type $3$ or $4$ receive the same label as their parent.
		\item If a vertex $v$ of type $3$ or $4$ has offspring $v_1, \ldots, v_d$ with $d \ge 1$ then set $(a_0, \ldots, a_d) := \bm{\beta}_{\bm{T}(\kappa)}(v_i)$ and set $\ell(v_i) := \ell(v) + \sum_{j=0}^{i-1} a_j$ for all $1 \le i \le d$.
	\end{enumerate}
\end{enumerate}

This construction produces a so-called \emph{mobile}. We emphasize that in the second step we choose for any vertex $v$ of type $3$ or $4$ the vector $\bm{\beta}_{\bm{T}(\kappa)}(v)$ at random in a way that depends only on the ordered list of offspring vertices of $v$, their types, and the type of $v$ (since it determines the type of its parent). In combinatorial language, $(\bm{T}(\kappa), \bm{\beta}_{\bm{T}(\kappa)})$ is a special case of an \emph{multi-type enriched plane tree}. We refer to it as the \emph{canonical decoration} of $\bm{T}(\kappa)$.

\subsection{The Bouttier--Di Francesco--Guitter transformation}
\label{sec:bdg}




We let $\bm{T}^+$ denote an independent copy of $(\bm{T}(1), \bm{\beta}_{\bm{T}(1)})$. We let $\bm{T}^0$ denote the result of taking two independent copies of $(\bm{T}(2), \bm{\beta}_{\bm{T}(2)})$ and identifying their roots. 
Let $(T, \beta)$ be a possible finite outcome of $\bm{T}^+$ or $\bm{T}^0$, and let $(\ell(v))_{v \in T}$ denote the corresponding labels. The Bouttier--Di Francesco--Guitter transformation~\cite{MR2097335} associates a planar map $\Psi(T, \beta)$ to the decorated tree $(T, \beta)$ in such a way that
\begin{itemize}
	\item the number of vertices of the map equals $1 + \#_1T$,
	\item the number of edges of the map equals $\#_1 T + \#_3 T + \#_4 T -1$,
	\item and the number of faces of the map equals $\#_3 T + \#_4 T$.
\end{itemize}

The transformation $\Psi$ is as follows. We draw $T$ in the plane and order the corners according to the standard contour process that starts at the root vertex. Let $v_1, \ldots, v_p$ denote the ordered list of vertices of type $1$ or $2$ that we visit in the contour process. That is, a vertex gets visited multiple types according to the number of angular sectors around it.  We let $\ell_1, \ldots, \ell_p$ denote their labels. We extend these  lists cyclically, so that $v_{ip+k} = v_{k}$ for $i \ge 1$ and $1 \le k \le p$. We add an extra vertex $r$ with type $1$ outside of $T$ and let its label $\ell(r)$ be one less than the  minimum of labels of all type $1$ vertices. For each $1 \le i \le p$ we draw an arc between the vertex $v_i$ and its \emph{successor}.  If $v_i$ has type $1$ then the successor is the next corner in the cyclic list of type $1$ with label $\ell_i - 1$. If there is no such corner, then we let $r$ be the successor of $v_i$. Likewise, if $v_i$ has type $2$ then the successor of $v_i$ is the next corner of type $1$ with label $\ell_i - 1/2$, or $r$ if there is no such corner. It is possible to draw all arcs so that they only may intersect at end points. We now delete the original edges of the tree $T$, as well as all vertices of type $3$ and $4$. Vertices of type $2$ get erased as well, merging the corresponding pairs of arcs. We are left with a planar map having a marked vertex $r$. If the root of $T$ has type $1$ we let the root edge be the first arc that was drawn and have it point to the root of $T$. If the root of $T$ has type $2$ (and hence has precisely two children, both of type $4$), we let the root edge be the result of the merger of the two arcs incident to the root of $T$ and let it point towards the successor of the first corner  encountered in the contour process. Figure~\ref{fi:transfo} illustrates the transformation $\psi$ for an example.

\begin{figure}
	\centering
	\begin{minipage}{1.0\textwidth}
		\subfloat[A mobile from which we are going to construct a vertex-marked rooted planar map.]{
			\includegraphics[width=0.45\textwidth]{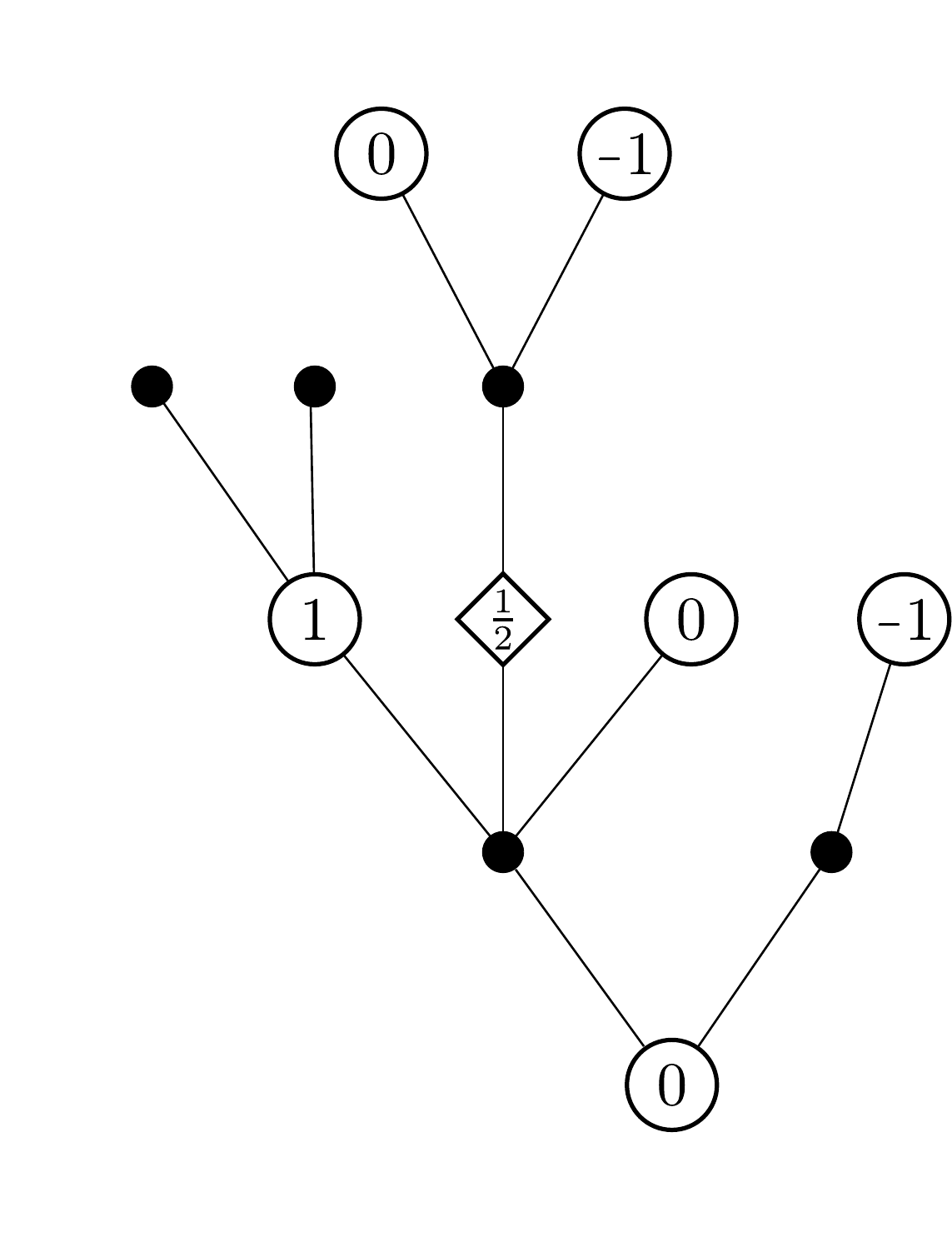}
		}
		\hspace{2em}
		\subfloat[Listing the corners incident to vertices of type $1$ and $2$. Adding a marked  vertex.]{
			\includegraphics[width=0.45\textwidth]{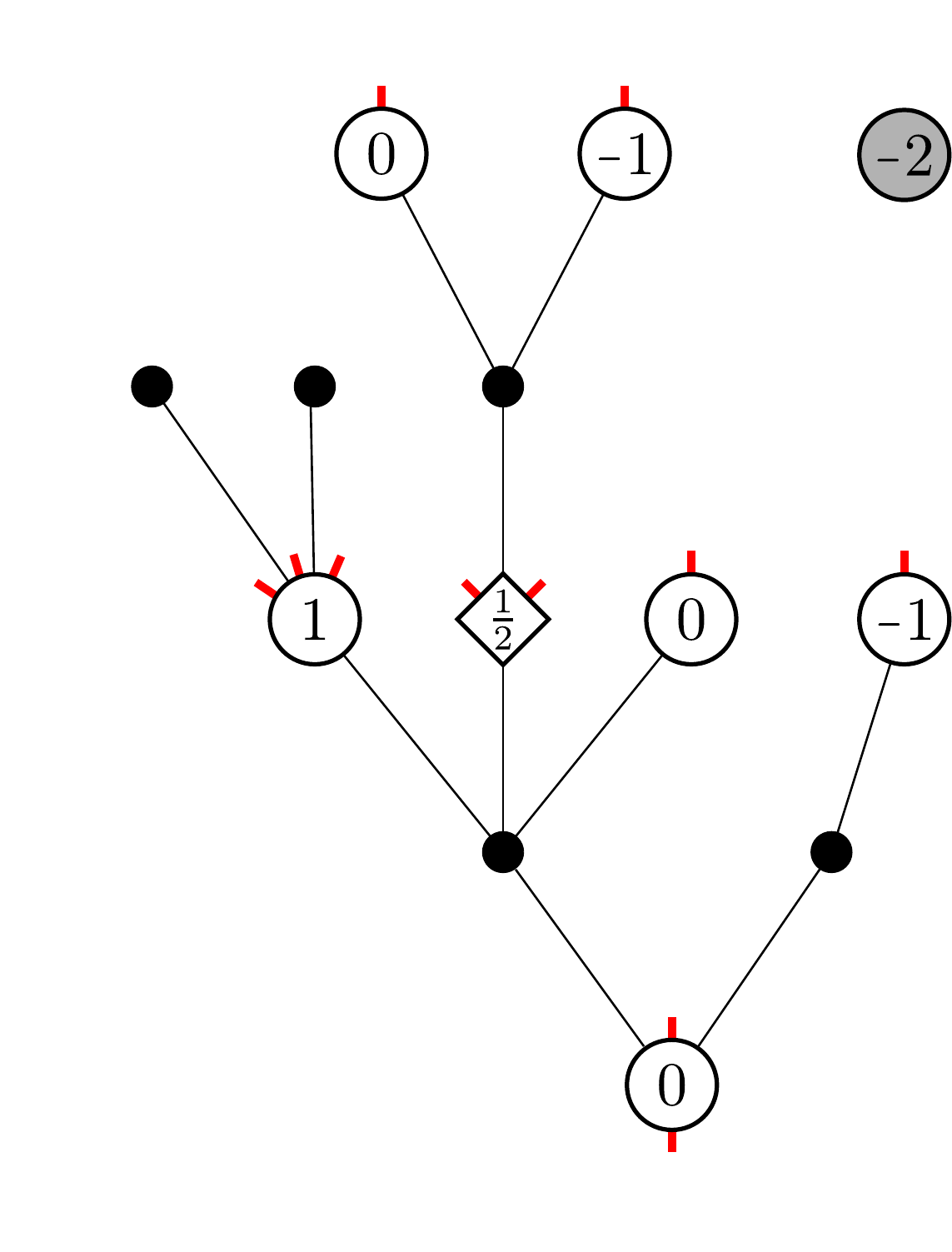}
		}
		\\
		\subfloat[Drawing arcs and distinguishing an oriented root edge.]{ 
			\includegraphics[width=0.45\textwidth]{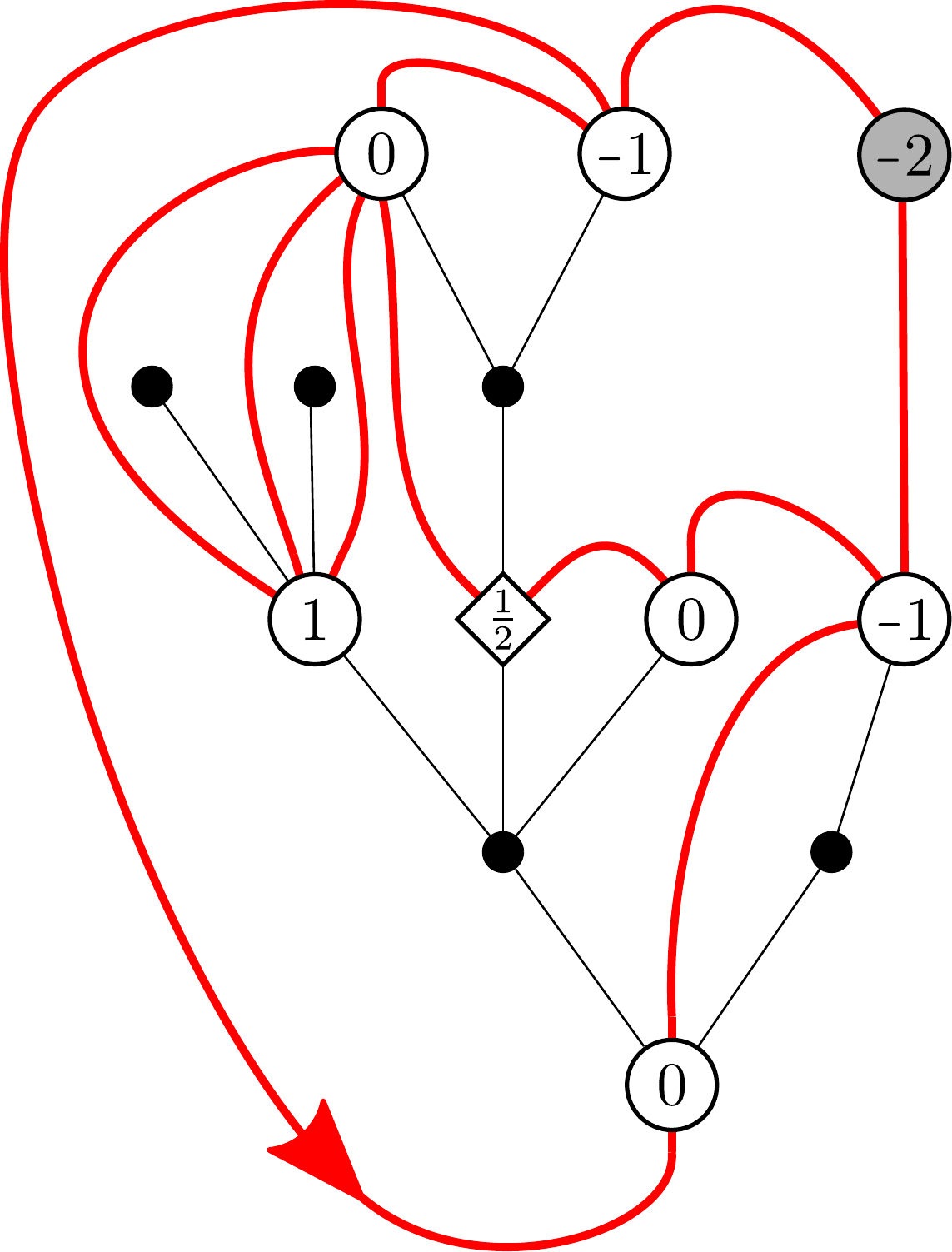}
		}
		\hspace{2em}
		\subfloat[Removing old edges, removing type $3$ and type $4$ vertices, and replacing type $2$ vertices by arcs.]{
			\includegraphics[width=0.45\textwidth]{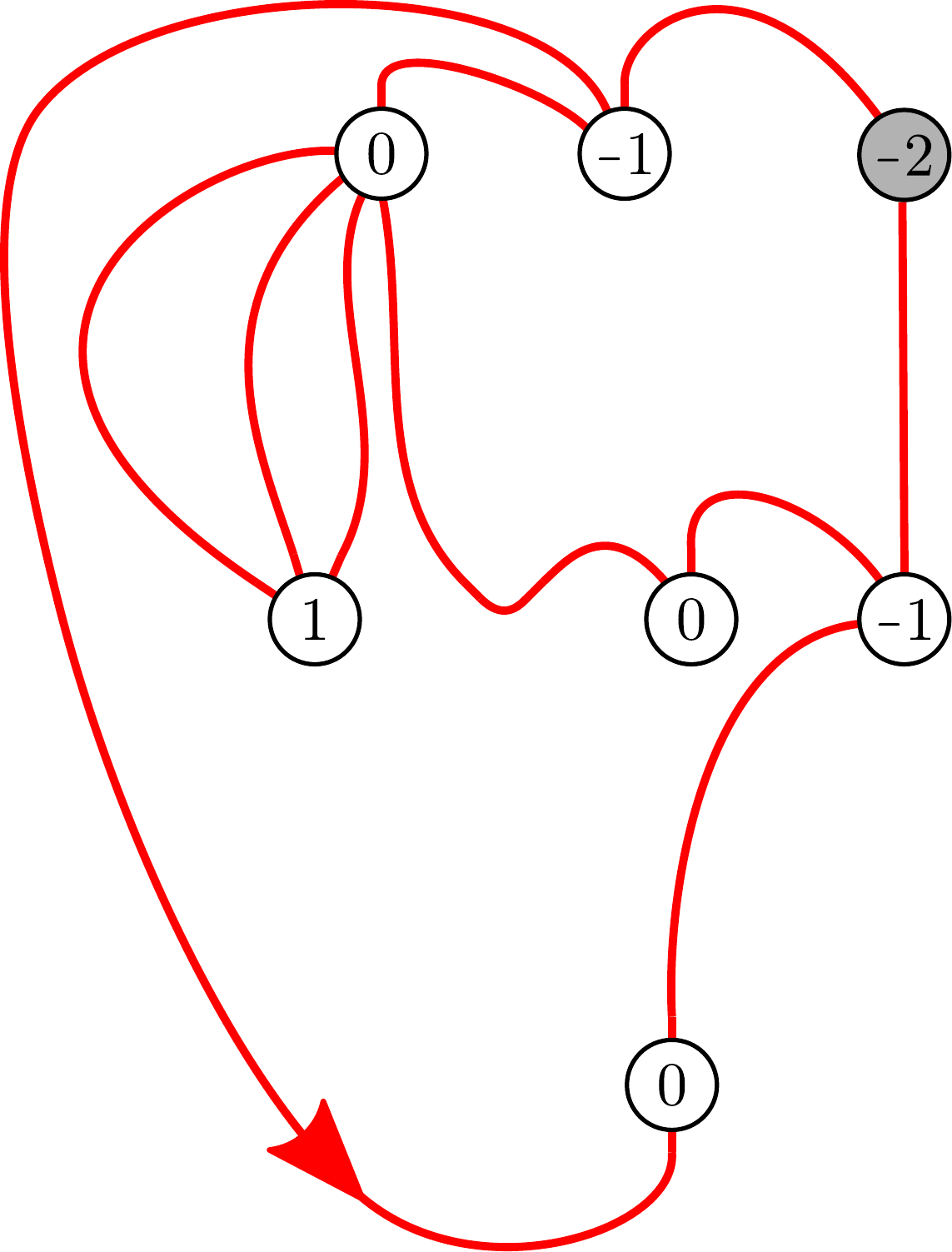}
		}
		\caption{The correspondence between mobiles and vertex-marked rooted planar maps.}
		\label{fi:transfo}
	\end{minipage}
\end{figure}

The Boltzmann distributed map $\mM$ is a mixture of the random maps $\mM^+$, $\mM^0$, and $\mM^-$ obtained by conditioning $\mM$ on belonging to $\cM^+$, $\cM^0$, and $\cM^-$.  As observed by \cite{MR2509622}, it holds that $\Psi(\bm{T}^+) \eqdist \mM^+$ and $\Psi(\bm{T}^0) \eqdist \mM^0$. Moreover, $\mM^{-}$ may be obtained from $\mM^+$ by reversing the direction of the root edge.

\subsection{Regimes of weight sequences}
\label{sec:regimes}

\cite[Prop. 1]{MR2509622} showed that the weight sequence $\bm{q}$ is admissible if and only if the system of equations
\begin{align}
	\label{eq:tc1}
	f^\bullet(x,y) &= 1 - \frac{1}{x} \\
	\label{eq:tc2}
	f^\diamond(x,y) &= y
\end{align}
has a solution $(x,y)$ with $x>1$ such that the matrix
\[
\begin{pmatrix}
	0 & 0 & x-1 \\
	\frac{x}{y}\partial_x f^\diamond(x,y) & \partial_y f^\diamond(x,y) & 0 \\
	\frac{x^2}{x-1} \partial_x f^\bullet(x,y) & \frac{xy}{x-1} \partial_y f^\bullet(x,y) & 0 \\
\end{pmatrix}
\]
has spectral radius smaller or equal to one. Any such solution $(x,y)$ necessarily satisfies 
\begin{align}
	(x,y) = (Z_{\bm{q}}^+, \sqrt{Z_{\bm{q}}^0}).
\end{align}
\cite[Def. 1]{MR2509622} termed an admissible weight sequence $\bm{q}$  \emph{critical}, if the spectral radius of this matrix is equal to $1$. This amounts to the condition
\begin{align}
	\label{eq:crit}
	x^2 J_f(x, y) +1 = x^2 \partial_x f^\bullet(x, y) + \partial_y f^\diamond(x, y),	
\end{align}
with $J_f$ denoting the (signed) Jacobian of the function $(f^\bullet, f^\diamond): \ndR_+^2 \to (\ndR_+ \cup \{\infty\})^2$.
It is termed \emph{regular critical}, if additionally 
\begin{align}
	f^\bullet( Z_{\bm{q}}^+ + \epsilon, y\sqrt{Z_{\bm{q}}^0} + \epsilon) < \infty
\end{align}
for some $\epsilon>0$. As was made explicit by~\cite{MR3769811}, this applies to various useful cases such as unrestricted maps or $p$-angulations for arbitrary $p \ge 3$. The irreducible offspring distribution $\bm{\xi}$ is critical (or regular critical) if and only if the weight sequence $\bm{q}$ is critical (or regular critical).  

\section{Quenched local convergence}

\begin{theorem}
	\label{te:regcritquench}
	Suppose that the weight-sequence $\bm{q}$ is regular critical. Let $\mM_n$ denote the $\bm{q}$-Boltzmann planar map, conditioned on either having $n$ vertices, or edges, or faces. Let $u_n$ denote either a uniformly selected vertex, half-edge, or face. There are integers $a \ge 0$ and $d \ge 1$ and a random infinite locally finite limit map $\hat{\mM}$ with finite face degrees such that, in the local topology for vertex-rooted or half-edge rooted or face-rooted planar maps, the conditional law $\mfL( (\mM_n, u_n) \mid \mM_n)$ satisfies
	\begin{align}
		\label{eq:quenched}
		\mfL( (\mM_n, u_n) \mid \mM_n) \convp \mathfrak{L}(\hat{\mM})
	\end{align}
	as $n \in a + d \ndZ$ tends to infinity.
\end{theorem}
Of course, the limit object differs depending on which conditioning we choose and which type of marking we select. The quenched limit~\eqref{eq:quenched} implies the annealed convergence
\begin{align}
	\label{eq:annealed}
	(\mM_n, u_n) \convdis \hat{\mM}
\end{align}
by dominated convergence. If $u_n$ denotes a uniformly selected half-edge, then \eqref{eq:annealed} is the annealed convergence established by~\cite[Thm. 6.1]{MR3769811} (see also \cite{MR2013797,2005math.....12304K,MR3183575,MR3083919,MR3256879}), who only required criticality in the case where $\mM_n$ is the Boltzmann map with $n$ vertices. \cite{DRMOTA2020108666} described a general method for deducing limits for the vicinity of random vertices if a limit for the vicinity of a random corners is known. The method applies to regular critical Boltzmann planar maps and other settings. Obtaining an explicit description of the limit was left as an open question in \cite{DRMOTA2020108666}, and the construction of the limit from an infinite mobile with a backwards growing spine the proof of Theorem~\ref{te:regcritquench} resolves this question in the present setting.

Note that, as was shown by~\cite[Sec. 6.3.5]{MR3769811}, in the present setting the total variational distance between $\mM_n$ (a corner-rooted map with an additional marked vertex, not to be confused with $u_n$) and a $\bm{q}$-Boltzmann map $\tilde{\mM}_n$ without a marked vertex tends to zero as $n$ becomes large:
\begin{align}
	\label{eq:deroot}
	\lim_{n \to \infty} d_{\textsc{TV}}(\mM_n, \tilde{\mM}_n) = 0.
\end{align}
Hence Theorem~\ref{te:regcritquench} also holds for~$\tilde{\mM}_n$.

\subsection{Proof strategy}
The existence of $a \ge 0$ and $d \ge 1$ (which depend on the form of conditioning we use), so that $\mM_n$ is well-defined for $n \in a + d\ndZ$ large enough, was shown by~\cite[Lem. 6.1]{MR3769811}. Let $\bm{\gamma}\in \ndN_0^4$ be either $(1,0,0,0)$ or $(1,0,1,1)$ or $(0,0,1,1)$, depending on whether we condition on the number of vertices, edges, or faces. We also set $\mathfrak{G}_0 = \{1\}$ or $\mathfrak{G}_0 = \{1,3,4\}$ or $\mathfrak{G}_0 = \{3,4\}$ accordingly.

Recall that $\bm{T}^+$ denotes an independent copy of $(\bm{T}(1), \bm{\beta}_{\bm{T}(1)})$, and  $\bm{T}^0$ is the result of taking two independent copies of $(\bm{T}(2), \bm{\beta}_{\bm{T}(2)})$ and identifying their roots.
Recall also that the Boltzmann distributed map $\mM$ is a mixture of the random maps  $\Psi(\bm{T}^+) \eqdist \mM^+$, $\Psi(\bm{T}^0) \eqdist \mM^0$, and the result $\mM^{-}$ of reversing the direction of the root-edge $\mM^+$.

Throughout the entire proof, a subscript $n$ of a random tree  denotes that we condition the tree on the event $|\cdot|_{\bm{\gamma}} = n$ if $\bm{\gamma} =  (0,0,1,1)$,  $|\cdot|_{\bm{\gamma}} = n-1$ if $\bm{\gamma} = (1,0,0,0)$, and $|\cdot|_{\bm{\gamma}} = n+1$ if $\bm{\gamma} = (1,0,1,1)$. A subscript $n$ of a random map will  denote that we condition the map accordingly on having $n$ faces or vertices or edges.

Let $\kappa \in \{1, \ldots, 4\}$ be a type. If we select a vertex $v_n$ from $\bm{T}_n(\kappa)$ with type in $\mathfrak{G}_0$ uniformly at random, then by~\cite[Thm. 6]{multrerooted} 
\begin{align}
	\label{eq:untree}
	\mfL( (\bm{T}_n(\kappa), v_n) \mid \bm{T}_n(\kappa)) \convp \mathfrak{L}(\hat{\bm{T}}(\eta))
\end{align}
for a random type $\eta$ that only depends on $\bm{\xi}$ and $\bm{\gamma}$ (and not on $\kappa$). Adding canonical decorations, this implies
\begin{align}
	\label{eq:entree}
	\mfL( (\bm{T}_n(\kappa), \bm{\beta}_{\bm{T}_n(\kappa)}, v_n) \mid \bm{T}_n(\kappa)) \convp \mathfrak{L}(\hat{\bm{T}}(\eta),\bm{\beta}_{\hat{\bm{T}}(\eta)}).
\end{align}
(See also \cite{zbMATH07235577} for a general theory of limits and fringe distributions of random decorated or enriched trees.)

We are going to show that:
\begin{enumerate}[\qquad a)]
	\item The decorated tree $(\hat{\bm{T}}(\eta),\bm{\beta}_{\hat{\bm{T}}(\eta)})$ corresponds to an infinite vertex-, corner-, or face-rooted map $\hat{\mM}$ via an extension of the Bouttier--Di Francesco--Guitter transformation.
	\item The convergence \eqref{eq:entree}  implies
	\begin{align}
		\label{eq:togo}
		\mfL( (\mM_n^+, u_n) \mid \mM_n^+) \convp \mathfrak{L}(\hat{\mM}).
	\end{align}
	\item Convergence of $\mM_n^-$ follows from \eqref{eq:togo}  and $\mM_n^0$ may be treated analogously as $\mM_n^+$.
\end{enumerate} 

Having these intermediate results at hand, Theorem~\ref{te:regcritquench} immediately follows. In the following subsection, we verify the three claims individually.

\subsection{Claim a)}
In the third step of the procedure given in Section~\ref{sec:mobiles} we described a process for transforming the decorations into labels. We cannot apply this process directly to $(\hat{\bm{T}}(\eta),\bm{\beta}_{\hat{\bm{T}}(\eta)})$ since the tree has an infinite backwards growing spine of ancestors instead of a root. However, if we assign any valid label to a single vertex $v$  (with value in $\ndZ$ if $v$ has type $1$ or $3$ and value in $\frac{1}{2} + \ndZ$ if $v$ has type $2$ or $4$), then the decorations determine the labels of all other vertices. Moreover, the differences in the labels between any pair of vertices does not depend on the label we started with. Hence let us assign a valid label $0$ or $1/2$ to the marked vertex of $(\hat{\bm{T}}(\eta),\bm{\beta}_{\hat{\bm{T}}(\eta)})$ (depending on whether its type $\eta$ lies in $\{1,3\}$ or $\{2,4\}$), and extend this in a unique way according to the decorations to labels~$(\ell(v))_{v \in \hat{\bm{T}}(\eta)}$. 


\begin{lemma}
	The labels of the type $1$ ancestors of the marked vertex in $\hat{\bm{T}}(\eta)$ have almost surely no lower bound.
\end{lemma}
\begin{proof}
	First, let us observe that
	\begin{align}
		\label{eq:add1}
		\hat{\bm{T}}(\eta) \eqdist \hat{\bm{T}}(1,\eta).
	\end{align}
	This could be verified directly, or as follows: The limit in Equation~\eqref{eq:untree} is a special case of~\cite[Thm. 6]{multrerooted}, which was obtained as an application of the more general theorem~\cite[Thm. 1]{multrerooted}. We could  just as well have applied~\cite[Thm. 2, Rem. 2]{multrerooted} instead, yielding that~\eqref{eq:untree} holds with $\hat{\bm{T}}(1,\eta)$ instead of $\hat{\bm{T}}(\eta)$. This verifies~\eqref{eq:add1}.
	
	Let $u_1, u_2, \ldots$ denote the list of type $1$ ancestors of the marked vertex in $\hat{\bm{T}}(\eta)$ (excluding the marked vertex itself, if it has type $1$), so that $u_{i+1}$ is an ancestor of $u_i$ for all $i \ge 1$. Then the family of differences of labels $\ell(u_{i+1}) - \ell(u_i)$, $i \ge 1$ are independent and identically distributed. The distribution is given by forming the canonical decoration of $\hat{\bm{T}}^1$, assigning labels accordingly with an arbitrary starting value for the root of $\hat{\bm{T}}^1$, and forming the difference of the labels between the root and the marked leaf of  $\hat{\bm{T}}^1$. Thus, the labels $(\ell(u_i))_{i \ge 1}$ form a random walk  with i.i.d. steps and a random starting value $\ell(u_1)$. 
	
	It is known that this random walk is centred: Indeed, consider the local weak limit $\tilde{\bm{T}}$  of $\bm{T}_n(1)$ established by~\cite{MR3769811}, that describes the asymptotic vicinity of the root (and not a random location) of $\bm{T}_n(1)$. The construction of $\tilde{\bm{T}}$ is as follows. We start with  a type $1$ vertex that gets identified with an independent copy of  $\hat{\bm{T}}^1$. All non-marked type $1$ leaves become roots of independent copies of $\bm{T}(1)$. For the marked leaf, we proceed recursively in the same way as for the root (identifying it with the root of a fresh independent copy of $\hat{\bm{T}}^1$, and so on). Hence $\bm{T}(1)$ has an infinite spine, obtained by concatenating independent copies of $\bm{T}(1)$. In particular, if we form the canonical decoration of $\tilde{\bm{T}}$  and assign labels accordingly (with, say, a starting value $0$ for the root vertex), then the labels of the type $1$ vertices of the spine form a random walk with i.i.d. steps and the \emph{same} step distribution as for the random walk $(\ell(u_i))_{i \ge 1}$. \cite[Proof of Lem. 6.5]{MR3769811} showed that this step distribution has average value $0$. Hence, analogously as for \cite[Lem. 6.5]{MR3769811}, it follows from~\cite[Thm. 9.2]{zbMATH01713116} that almost surely
	\[
	\inf_{i \ge 1} \ell(u_i) = -\infty.
	\]
	This completes the proof.
\end{proof}

We may order the corners $(c_i)_{i \in \ndZ}$ incident to vertices of type $1$ or $2$ of $\hat{\bm{T}}(\eta)$ such that for all $i \in \ndZ$ the corner $c_{i+1}$ is the successor of $c_i$ in the clock-wise contour exploration.
This allows us to  canonically  extend the Bouttier--Di Francesco--Guitter transformation from Section~\ref{sec:bdg} to assign an infinite locally finite planar map $\hat{\mM}$ to the infinite labelled tree $(\hat{\bm{T}}(\eta), (\ell(v))_{v \in \hat{\bm{T}}(\eta)})$. Here we do not have to add an additional marked vertex, because the labels of type $1$ vertices along the backwards growing spine of $\hat{\bm{T}}(\eta)$ have no lower bound. By construction, all faces of $\hat{\mM}$ have finite degree. 

Depending on whether $u_n$ is a random vertex, half-edge, or face of $\mM_n$, we mark $\hat{\mM}$ as follows. Let $w$ denote the marked vertex of $\hat{\bm{T}}(\eta)$, which has an infinite number of ancestors. In the vertex case, $w$ has type $1$ and corresponds canonically to a vertex of $\hat{\mM}$. We consider $\hat{\mM}$ as rooted at this vertex. In the face case, $w$ has type $3$ or $4$ and corresponds canonically to a face. In this case, we consider $\hat{\mM}$ as rooted at this face. In the half-edge case, $w$ has type $1$, $3$, or $4$ and corresponds canonically to an edge, which we orient according to an independent fair coin flip. In detail: If $w$ has type $4$, then it is the only child of a non-root type $2$ vertex that corresponds to the edge obtained by joining the arcs drawn at its two corners. Hence $w$ corresponds canonically to this edge. If $w$ has type $1$, then each of its corners corresponds to the edge we drew when visiting this corner in the contour exploration. The number of these corners equals $1$ plus the number of offspring vertices, all of which have type $3$. Hence  $w$ and its children correspond bijectively to the arcs we drew starting at a corner of $w$. In particular, $w$ corresponds canonically to an arc. Likewise, if $w$ has type $3$ it also corresponds canonically to an edge that we drew starting at a corner of its type $1$ parent.

This verifies Claim a).

\subsection{Claim b)}
Suppose that $\kappa=1$. The vertex $v_n$ of $(\bm{T}_n(\kappa), \beta_{\bm{T}_n(\kappa)})$ corresponds similarly to a marked vertex or face or half-edge $u_n'$ of~$\mM_n^+$. Modifications in the correspondence may be required when $v_n$ or its parent is the root of $\bm{T}_n(\kappa)$, but the probability for this event tends to zero and hence we may safely ignore this. Furthermore, $(\mM_n^+, u_n)$ and $(\mM_n^+, u_n')$ may not follow the same distribution (for example, when $u_n$ is a uniform vertex, then $u_n'$ is a uniform \emph{non-marked} vertex, as $u_n'$ is never equal to the additional vertex we added in the BDFG bijection). However, it is clear that there is an event (that depends on $n$) whose probability tends to $1$ as $n$ becomes large, such that $(\mM_n^+, u_n)$ and $(\mM_n^+, u_n')$ are identically distributed when conditioned on this event. Hence we may also safely ignore the difference between $u_n$ and $u_n'$. Using the continuous mapping theorem, it hence follows from~\eqref{eq:entree} that 
\begin{align}
	\label{eq:same}
	\mfL( (\mM_n^+, u_n) \mid \mM_n^+) \convp \mathfrak{L}(\hat{\mM}).
\end{align}
This verifies Claim b).

\subsection{Claim c)} The same convergence as in~\eqref{eq:same} follows immediately for $\mM_n^-$, since the vicinity of a random point is not affected by the orientation of the root edge. As for $\mM_n^0$,  it follows from~\cite[Prop 2.2]{MR3769811} that $|\bm{T}(2)|_{\bm{\gamma}}$ takes only values from a shifted lattice, and has a density that varies regularly with index $-3/2$ along that shifted lattice. It follows that if we condition independent copies $\bm{S}^{(1)}$ and $\bm{S}^{(2)}$ of $\bm{T}(2)$ on the event $|\bm{S}^{(1)}|_{\bm{\gamma}} + |\bm{S}^{(2)}|_{\bm{\gamma}} = n$ then
\begin{align}
	\lim_{n \to \infty} \min( (|\bm{S}^{(1)}|_{\bm{\gamma}}, |\bm{S}^{(2)}|_{\bm{\gamma}}) \mid |\bm{S}^{(1)}|_{\bm{\gamma}} + |\bm{S}^{(2)}|_{\bm{\gamma}} = n) \convdis |\bm{T}(2)|_{\bm{\gamma}}.
\end{align}
This may easily be verified elementarily or be viewed as a special case for results on general models of random partitions, see~\cite[Thm. 3.4, Prop 2.5]{doi:10.1002/rsa.20771}. Consequently, all but a negligible number of vertices whose extended fringe subtree has a certain shape will lie in a giant component with size (``size'' referring to $|\cdot|_{\bm{\gamma}}$) $m - O_p(1)$. If we let $\bm{S}$ denote the result of  identifying the roots of $\bm{S}^{(1)}$ and $\bm{S}^{(2)}$ and let $w_n$ denote a uniformly selected vertex of the conditioned tree $\bm{S}_n$ with type in $\mathfrak{G}_0$, then it follows by~\eqref{eq:untree} that
\begin{align}
	\mfL( (\bm{S}_n, w_n) \mid \bm{S}_n) \convp \mathfrak{L}(\hat{\bm{T}}(\eta)).
\end{align}
(Recall that above we assigned a clear meaning to all occurrences of $n$ as a subscript of a random tree, making $\bm{S}_n$ a conditioned version of $\bm{S}$ that depends on $\bm{\gamma}$.)
Hence, adding canonical decorations,
\begin{align}
	\label{eq:entree2}
	\mfL( (\bm{S}_n, \bm{\beta}_{\bm{S}_n}, w_n) \mid \bm{S}_n) \convp \mathfrak{L}(\hat{\bm{T}}(\eta),\bm{\beta}_{\hat{\bm{T}}(\eta)}).
\end{align}
Thus quenched convergence of $\mM_n^0$ towards $\hat{\mM}$ may be deduced in exactly the same way using the mapping theorem as for $\mM_n^+$, only instead of using Equation~\eqref{eq:entree} we use Equation~\eqref{eq:entree2}. This verifies Claim c).

\section{Random planar maps with vertex weights}

Let $t>0$ be a constant. We let $\mM_n^t$ denote a random planar map with $n$ edges that assumes a map $M$ (with $n$ edges)  with probability proportional to $t^{\ve(M)}$.
\begin{theorem}
	\label{te:quenchmaptt}
	The random map $\mM_n^t$ admits a distributional limit~$\hat{\mM}^t$ in the local topology. Letting  $c_n$ denote a uniformly selected corner of $\mM_n^t$, it holds that
	\begin{align}
		\label{eq:toshow}
		\mfL( (\mM_n^t, c_n) \mid \mM_n^t) \convp \mfL(\hat{\mM}^t).
	\end{align}
\end{theorem}
\begin{proof}
	For any $\lambda>0$ we may consider the weights
	\begin{align}
		q_n = t \lambda^n, \qquad n \ge 1.
	\end{align}
	This way, any map with $n$ edges and $m$ faces receives weight $\lambda^{2n} t^m$. We are going to argue below that for any $t>0$ we may choose $\lambda$ so that $\bm{q}=(q_n)_{n \ge 1}$ is regular critical.  By elementary identities of power series (compare with \cite[Proof of Prop. 6.3]{MR3769811}) the expressions in Equations~\eqref{eq:do1} and \eqref{eq:do2} simplify to
	\begin{align}
		\label{eq:c1}
		f^\bullet(x,y) &= \frac{t(1- Z)}{2 x Z}, \\
		\label{eq:c2}
		f^\diamond(x,y) &= \frac{t \lambda}{(1 - \lambda y)Z},
	\end{align}
	with
	\begin{align}
		\label{eq:origz}
		Z := \sqrt{1-\frac{4\lambda^2 x}{(1 - \lambda y)^2}}.
	\end{align}
	Conditions~\eqref{eq:tc1} and~\eqref{eq:tc2} may be rephrased by
	\begin{align}
		\label{eq:z1}
		Z = \frac{t}{t+2 x-2},
	\end{align}
	and
	\begin{align}
		Z = -\frac{\lambda  t}{y (\lambda  y-1)}.
	\end{align}
	Note that this implies $x>1$. Combining the last two equalities, we obtain
	\begin{align}
		\lambda =\frac{y}{t+2 x+y^2-2}.
	\end{align}
	Plugging this expression into Equations~\eqref{eq:tc1} and \eqref{eq:tc2} and noting that~\eqref{eq:z1} implies $x>1$ yields
	\begin{align}
		\label{eq:ylfinal}
		y=\frac{\sqrt{x-1} \sqrt{t+x-1}}{\sqrt{x}} \quad \text{and} \quad  
		\lambda =\frac{\sqrt{x-1} \sqrt{x} \sqrt{t+x-1}}{2 (t-2) x-t+3 x^2+1}  \quad \text{and} \quad   x>1.
	\end{align}
	Moreover, for any triple $(x,y,\lambda)$ of real numbers satisfying~\eqref{eq:ylfinal}, we may easily verify that  Equations~\eqref{eq:tc1} and \eqref{eq:tc2} hold (and that $y>0$ and $\lambda>0$). Plugging~\eqref{eq:ylfinal} into the criticality condition~~\eqref{eq:crit} yields the complicated expression
	{
		\small
		\begin{align}
			x = &\frac{2}{3}-\frac{t}{3} + \frac{1}{6} \sqrt{6 \sqrt[3]{2} \sqrt[3]{-(t-1)^2 t^2}+4 (t-1) t+4} \\
			&+ \frac{1}{2} \sqrt{-\frac{4 (t+1) (2 t-1) (t-2)}{9 \sqrt{\frac{3 \sqrt[3]{-(t-1)^2 t^2}}{2^{2/3}}+(t-1) t+1}}-\frac{2}{3} \sqrt[3]{2} \sqrt[3]{-(t-1)^2
					t^2}+\frac{8}{9} (t-2)^2+\frac{8 (t-1)}{3}}. \nonumber
		\end{align}
	}
	This solution is strictly bigger than $1$ for any $t>0$ and defining $y$ and $\lambda$ according to~\eqref{eq:ylfinal} we obtain a solution to Equations~\eqref{eq:tc1},~\eqref{eq:tc2}, and~\eqref{eq:crit}. Hence for this choice of $\lambda$ (depending on $t$) the weight sequence $\bm{q}$ is critical. It is clear from the expressions~\eqref{eq:c1},~\eqref{eq:c2},~\eqref{eq:origz} that $\bm{q}$ is even regular critical in this case.

	Let $\mM_n$ denote the corresponding regular critical $\bm{q}$-Boltzmann planar map with $n$ edges. Let $u_n$ denote a uniformly selected corner of $\mM_n$. As $\bm{q}$ is regular critical, it follows by Theorem~\ref{te:regcritquench} that there is an infinite random planar map $\hat{\mM}$ with finite face degrees such that
	\begin{align}
		\mfL( (\mM_n, u_n) \mid \mM_n) \convp \mathfrak{L}(\hat{\mM}).
	\end{align}
	By~\eqref{eq:deroot}, the $\bm{q}$-Boltzmann map $\tilde{\mM}_n$ without a marked vertex consequently satisfies as well
	\begin{align}
		\label{eq:quenchedprel}
		\mfL( (\tilde{\mM}_n, u_n) \mid \tilde{\mM}_n) \convp \mathfrak{L}(\hat{\mM}).
	\end{align}
	
	The random planar map $\tilde{\mM}_n$ assumes any planar map $M$ with $n$ edges, $m$ faces and $k$ vertices with probability proportional to $t^{m}$. That is, 
	\begin{align}
		\Pr{\tilde{\mM}_n = M} = t^{m} c_{n,t}
	\end{align}
	for some constant $c_{n,t}>0$ that only depends on $n$ and $t$. Euler's formula entails that $m = 2 + n - k$. Hence
	\begin{align}
		\Pr{\tilde{\mM}_n = M} = t^{-k} c_{n,t}',
	\end{align}
with $c_{n,t}' = c_{n,t} t^{2+n}$ again only depending on $n$ and $t$. Thus,
\begin{align}
	\tilde{\mM}_n \eqdist \mM_n^{-t}.
\end{align}
Replacing $t$ by $t^{-1}$, Equation~\eqref{eq:toshow} now follows from Equation~\eqref{eq:quenchedprel}.
\end{proof}

\section*{Acknowledgement}

I warmly thank the associate editor and the referee for the thorough reading and helpful comments. In particular, for pointing out a simplification of the proof of Theorem~\ref{te:quenchmaptt}.


\begin{thebibliography}{10}
	
	\bibitem{MR1102319}
	D.~Aldous.
	\newblock Asymptotic fringe distributions for general families of random trees.
	\newblock {\em Ann. Appl. Probab.}, 1(2):228--266, 1991.
	
	\bibitem{MR2013797}
	O.~Angel and O.~Schramm.
	\newblock Uniform infinite planar triangulations.
	\newblock {\em Comm. Math. Phys.}, 241(2-3):191--213, 2003.
	
	\bibitem{MR1871555}
	C.~Banderier, P.~Flajolet, G.~Schaeffer, and M.~Soria.
	\newblock Random maps, coalescing saddles, singularity analysis, and {A}iry
	phenomena.
	\newblock {\em Random Structures Algorithms}, 19(3-4):194--246, 2001.
	\newblock Analysis of algorithms (Krynica Morska, 2000).
	
	\bibitem{MR3183575}
	J.~E. Bj{\"o}rnberg and S.~{\"O}. Stef{\'a}nsson.
	\newblock Recurrence of bipartite planar maps.
	\newblock {\em Electron. J. Probab.}, 19:no. 31, 40, 2014.
	
	\bibitem{MR2097335}
	J.~Bouttier, P.~Di~Francesco, and E.~Guitter.
	\newblock Planar maps as labeled mobiles.
	\newblock {\em Electron. J. Combin.}, 11(1):Research Paper 69, 27, 2004.
	
	\bibitem{zbMATH06549737}
	T.~{Budd}.
	\newblock {The peeling process of infinite Boltzmann planar maps.}
	\newblock {\em {Electron. J. Comb.}}, 23(1):research paper p1.28, 37, 2016.
	
	\bibitem{MR2735332}
	G.~Chapuy, E.~Fusy, O.~Gim\'{e}nez, and M.~Noy.
	\newblock On the diameter of random planar graphs.
	\newblock In {\em 21st {I}nternational {M}eeting on {P}robabilistic,
		{C}ombinatorial, and {A}symptotic {M}ethods in the {A}nalysis of {A}lgorithms
		({A}of{A}'10)}, Discrete Math. Theor. Comput. Sci. Proc., AM, pages 65--78.
	Assoc. Discrete Math. Theor. Comput. Sci., Nancy, 2010.
	
	\bibitem{curienlecture2}
	N.~Curien.
	\newblock Peeling random planar maps.
	\newblock {\em Saint-Flour course}
	\newblock {https://www.dropbox.com/s/bfjbuxiv4ms1gdl/StFlour.pdf}, 2016.

	
	\bibitem{MR3083919}
	N.~Curien, L.~M{\'e}nard, and G.~Miermont.
	\newblock A view from infinity of the uniform infinite planar quadrangulation.
	\newblock {\em ALEA Lat. Am. J. Probab. Math. Stat.}, 10(1):45--88, 2013.
	
	\bibitem{DRMOTA2020108666}
	M.~Drmota and B.~Stufler.
	\newblock Pattern occurrences in random planar maps.
	\newblock {\em Statistics \& Probability Letters}, 158:108666, 2020.
	
	
	\bibitem{MR3068033}
	O.~Gim\'{e}nez, M.~Noy, and J.~Ru\'{e}.
	\newblock Graph classes with given 3-connected components: asymptotic
	enumeration and random graphs.
	\newblock {\em Random Structures Algorithms}, 42(4):438--479, 2013.
	
	\bibitem{MR3342658}
	S.~Janson and S.~{\"O}. Stef{\'a}nsson.
	\newblock Scaling limits of random planar maps with a unique large face.
	\newblock {\em Ann. Probab.}, 43(3):1045--1081, 2015.
	
	\bibitem{zbMATH01713116}
	O.~{Kallenberg}.
	\newblock {Foundations of modern probability. 2nd ed}.
	\newblock {\em New York, NY: Springer}, 2nd ed. edition, 2002.
	
	\bibitem{2005math.....12304K}
	M.~{Krikun}.
	\newblock {Local structure of random quadrangulations}.
	\newblock {\em arXiv math/0512304}, Dec. 2005.
	
	\bibitem{MR3256879}
	L.~M{\'e}nard and P.~Nolin.
	\newblock Percolation on uniform infinite planar maps.
	\newblock {\em Electron. J. Probab.}, 19:no. 79, 27, 2014.
	
	\bibitem{MR2509622}
	G.~Miermont.
	\newblock An invariance principle for random planar maps.
	\newblock In {\em Fourth {C}olloquium on {M}athematics and {C}omputer {S}cience
		{A}lgorithms, {T}rees, {C}ombinatorics and {P}robabilities}, Discrete Math.
	Theor. Comput. Sci. Proc., AG, pages 39--57. Assoc. Discrete Math. Theor.
	Comput. Sci., Nancy, 2006.
	
	\bibitem{MR3769811}
	R.~Stephenson.
	\newblock Local convergence of large critical multi-type {G}alton-{W}atson
	trees and applications to random maps.
	\newblock {\em J. Theoret. Probab.}, 31(1):159--205, 2018.
	
	\bibitem{doi:10.1002/rsa.20771}
	B.~Stufler.
	\newblock Gibbs partitions: The convergent case.
	\newblock {\em Random Structures \& Algorithms}, 53(3):537--558, 2018.
	
	\bibitem{planar}
	B.~Stufler.
	\newblock {Local convergence of random planar graphs}.
	\newblock {\em arXiv:1908.04850}, 2019.
	
	\bibitem{zbMATH07235577}
	B.~{Stufler}.
	\newblock {Limits of random tree-like discrete structures}.
	\newblock {\em {Probab. Surv.}}, 17:318--477, 2020.
	
	\bibitem{multrerooted}
	B.~Stufler.
	\newblock {Rerooting multi-type branching trees: the infinite spine case}.
	\newblock {\em J. Theoret. Probab.}, to appear.
	
\end{thebibliography}

%
%

\end{document}